\documentclass[apaper,11pt]{article}

\usepackage{mathrsfs}
\usepackage{authblk}
\usepackage{color}
\usepackage{amsmath,  cite}
\usepackage{amsthm}
\usepackage{amssymb}
\usepackage{amsfonts}
\usepackage{latexsym}
\usepackage{multirow}
\usepackage{graphicx}
\usepackage{graphics}
\usepackage{epsfig} 
\usepackage{psfrag}
\usepackage{epstopdf}
\usepackage{caption}
\usepackage{float}

\usepackage{subcaption}
\usepackage{indentfirst}

\newcommand \Floor[1]
{
	\left	\lfloor {#1}\right \rfloor 
}

		\renewcommand{\theenumi}{\rm (\roman{enumi})}

\newtheorem{clm}{Claim}

\newcommand \Clm[2]
{
	\begin{clm}\label{#1}
		#2
	\end{clm}
}

 \def \setf {{\cal F}}
 
 \def \mid {{Mid}}

\def \iff {if and only if }

\newcommand {\setm}{{\cal M}}

\newtheorem{thm}{Theorem}[section]
\newtheorem{cor}[thm]{Corollary}
\newtheorem{Lem}{Lemma}[section]

\newtheorem{obs}{Observation}[section]
\newtheorem{pro}{Proposition}[section]
\newtheorem{prob}{Problem}[section]

\numberwithin{equation}{section}

\makeatletter \@addtoreset{equation}{section} \makeatother

\def \proofend {\hfill $\Box$}

\begin{document}

\title{
		Partial Domination of Middle Graphs
	}
	
	\author[1,2]{\small Shumin Zhang\footnote{Email: zhangshumin@qhnu.edu.cn.}}
	\author[1]{\small Minhui Li\footnote{Email: liminhuimath@163.com.}}
	\author[3\footnote{Corresponding author. Email: fengming.dong@nie.edu.sg, donggraph@163.com.}]{\small Fengming Dong}
	
	\affil[1]{\footnotesize 
		School of Mathematics and Statistics,
		Qinghai Normal University, Xining, 810001, China}
	
	\affil[2]{\footnotesize Academy of Plateau Science and Sustainability, People's Government of Qinghai Province and Beijing Normal University}
    \affil[3]{\footnotesize National Institute of Education, Nanyang Technological University, Singapore}
	\date{}
	\maketitle{}
	

\abstract{
For any graph $G=(V,E)$, a subset $S\subseteq V$ is called 
{\it an isolating set}
of $G$ 
if $V\setminus N_G[S]$ is an independent set 
of $G$, where $N_G[S]=S\cup N_G(S)$,
and {\it the isolation number} of $G$,
denoted by $\iota(G)$, is the size 
of a smallest isolating set of $G$.
In this article, we 
show that the isolation number of 
the middle graph of $G$
is equal to 
the size of a smallest maximal matching 
of $G$.
}




\baselineskip=0.65 cm

\section{Introduction}

All graphs considered in this article
are finite, undirected, nonempty and simple, and for standard graph theory terminology not given here we refer to \cite{W}. 
For a simple graph $G$, 
let $V(G)$ and $E(G)$ denote its vertex set and edge set.
For any $v\in V(G)$, 
the {\it open neighborhood} of $v$ in $G$, 
denoted by $N_G(v)$ (or simply $N(v)$),
is the set of neighbors of $v$ in $G$, and the {\it closed neighborhood} of
$v$ in $G$, denoted by $N_G[v]$, 
is defined to be $N_G(v)\cup\{v\}$.
The {\it degree} of $v$ in $G$, 
denoted by $d_G(v)$ (or simply $d(v)$),
is defined to be $|N_G(v)|$. 
If $d_G(v)=0$, then $v$ is an 
{\it isolate} vertex of $G$. 
If $d_G(v)=1$, then $v$ is called a {\it leaf} of $G$ 
and the only neighbor of a leaf is 
called a
{\it support vertex}.
The {\it minimum} and {\it maximum} degree of $G$ are respectively denoted by $\delta(G)$ and $\Delta(G)$.

 For any $S\subseteq V(G)$, the 
 {\it open neighborhood} of $S$ in $G$, 
 denoted by 
 $N_G(S)$ (or simply $N(S)$), 
 is $\{N_G(v): v\in S\}$, 
 and the {\it closed neighborhood} of $S$ in $G$, denoted by $N_G[S]$
 (or simply $N[S]$), 
 is $N(S)\cup\{S\}$. 
 Let $G[S]$ denote the subgraph of $G$ 
 induced by $S$, and $G-S$ the subgraph $G[V(G)\setminus S]$.  
 Given a graph $H$, a graph $G$ is 
 called $H$-{\it free} 
 if $G$ does not contain a subgraph 
 isomorphic to $H$.
 For any positive integer $n$ and $k$,
 let $P_n, C_n$ and $K_n$ denote the 
 path graph, cycle graph and 
 complete graph with $n$ vertices 
 respectively, and let $K_{n,k}$ 
 denote the complete bipartite graph 
 whose partite sets are of size $n$ and $k$ respectively.
 
 For any $S\subseteq V(G)$, 
 $S$ is called a {\it dominating set} of 
 $G$ if $N_G[S]=V(G)$.
 The {\it domination number} of $G$, denoted by $\gamma(G)$,
 is the minimum value of $|S|$ 
 over all dominating set $S$ of G. 
 For any set $\setf$ of graphs, 
 $S\subseteq V(G)$ is said to be 
 $\setf$-isolating if $G[V(G)\setminus N_G[S]]$ is $H$-free for each
 $H\in \setf$.
 Any subset $S$ of $V(G)$ 
 is called an {\it isolating set} of $G$
if  $V(G)\setminus N_G[S]$ 
is an independent set of $G$.
Clearly, $S$ is an isolating set of $G$
if and only if $S$ is a
$\{K_2\}$-isolating set.
The isolation number of $G$, 
denoted by $\iota(G)$, is 
the minimum value of $|S|$
over all isolating set $S$ of $G$.
Obviously, $\iota(G)\le \gamma(G)$,
as each dominating set of $G$ is also 
an isolating set of $G$.
For example, $\iota(P_5)=1$ while 
$\gamma(P_5)=2$. 

The study of isolation number of a graph
was initiated by Caro and Hansberg \cite{Caroa} in 2017.
They showed that  $\iota(G)\le\frac{n}{3}$ for any 
graph $G$ of order $n\ge6$
and $\iota(G)\le \frac{n}{4}$ for any  maximal outerplanar graph $G$ of order $n\ge4$. 
Tokunaga et al. \cite{Tokunaga} 
proved that for any maximal outerplanar graph $G$ of order $n$ with $n_2$ vertices of degree $2$, 
$\iota (G)\le \frac{n+n_2}{5}$ if $n_2\le \frac{n}{4}$, and $\iota (G)\le \frac{n-n_2}{3}$ otherwise. 
Lema\'{n}ska et al. \cite{MJ}
studied the relation between $\iota(T)$
and $\gamma(T)$ for a tree $T$
and showed that 
$\iota(T) = \frac n3$ implies that 
$\gamma(T)=\iota(T)$. 

Subdividing an edge $uv$ in a graph $G$ is a graph operation 
which  replaces edge $uv$ by 
a path of length $2$ joining $u$ and $v$.
The {\it middle} graph of  $G$,
denoted by $\mid(G)$,
 is the graph obtained from $G$ by subdividing each edge of $G$ exactly once and joining each pair of the added vertices on adjacent edges of $G$, as shown in Figure~\ref{MG}.

\begin{figure}[htbp]
	\centering
	
\includegraphics[width=7cm]{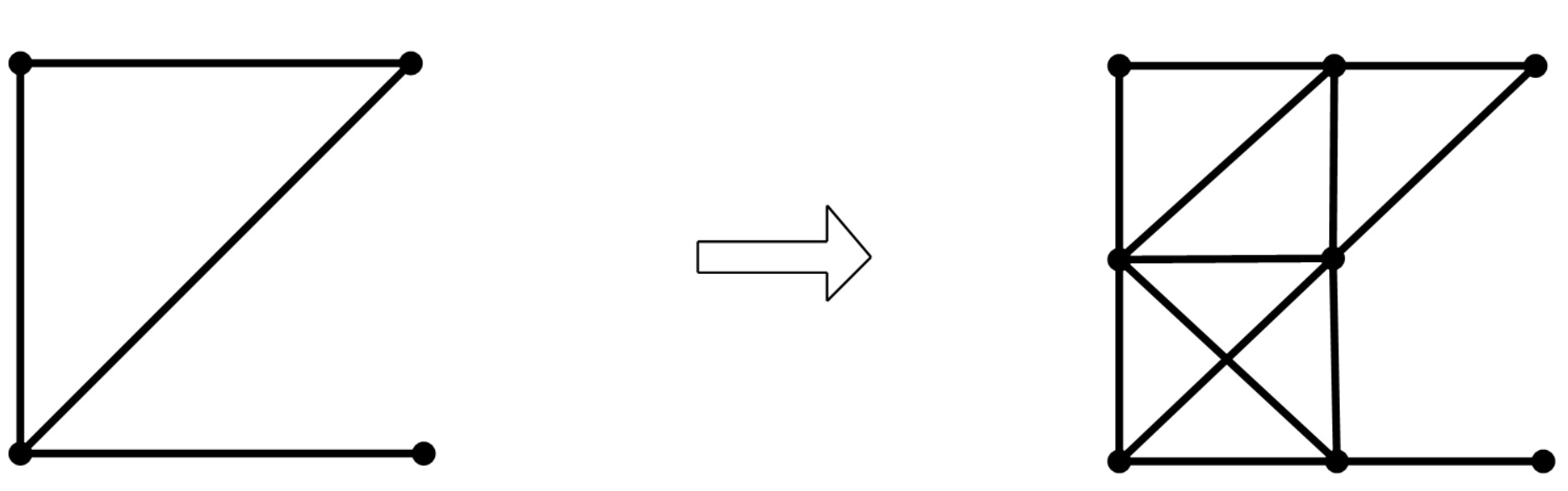}
	
	(a) $G$ \hspace{3.5 cm} (b) $Mid(G)$
	
	\caption{$G$ and its middle graph $\mid(G)$} 
	\label{MG}
\end{figure}

In the paper, we investigate the isolation numbers of middle graphs. 
In this article, we will establish the following conclusion 
on $\iota(\mid(G))$ 
for any graph $G$.

\begin{thm}\label{main1}
	For any graph $G$, $\iota(\mid(G))$ 
	is equal to $\nu'(G)$, 
	the minimum value of 
	$|M|$ over all maximal matching $M$ of $G$.
\end{thm}

By Theorem~\ref{main1}, 
 $\iota(G)=1$ if $G$ is a star. For the upper bounds of $\iota(G)$, we obtain the following conclusions. 

\begin{thm}\label{main2}
For any graph $G$ of order $n$,
$\iota(\mid(G))\le \frac n2$,
where the equality holds if and only if 
$G\cong K_{n}$ or $G\cong K_{n/2,n/2}$ 
when  $n$ is odd, and every maximal matching of $G$ is near-perfect otherwise. 
\end{thm} 

\begin{thm}\label{main3}
For any tree $T$ of order $n$, where $n\ge 3$,  
$\iota (\mid(T))\le \lfloor \frac{n-1}{2}\rfloor$,
where the upper bound is sharp.
\end{thm} 

The following sections are arranged 
as follows. 
For any $E_0\subseteq E(G)$,
let $V(E_0)=\{u,v: uv\in E_0\}$.
In Section 2, we show that 
$\iota(\mid(G))=\tau(G)$,
where $\tau(G)$ is the minimum value of 
$|E_0|$ over all subset $E_0$ of $E(G)$ 
such that $V(G)\setminus V(E_0)$ 
is an independent set of $G$, and then  prove that $\tau(G)$ is equal to 
the minimum value of $|M|$
over all maximal matching $M$ of $G$,
and thus Theorem~\ref{main1} follows.
In Section 3, 
we establish some result on the lower bounds and upper bounds of  $\iota(\mid(G))$, and in Sections 4 and 5, 
we proved Theorems~\ref{main2} and~\ref{main3} respectively.

\section{Proof of Theorem~\ref{main1}}

In this section, let $G=(V,E)$ be a graph with $V=\{v_1,\,v_2,\dots,v_n\}$ and  $V(\mid(G))=V\cup\mathcal{M}$,
where 
$\setm(G)=\{m_{i,j}:
v_iv_j\in E(G), 1\le i<j\le n \}$.
	For any $E_0\subseteq E(G)$, 
let $\setm(E_0)$ denote the set of vertices $m_{i,j}\in \setm$ 
such that  $v_iv_j\in E_0$.
Clearly, $|\setm(E_0)|=|E_0|$
and $|V(E_0)|\le 2|E_0|$.

	Let $\Theta(G)$ denote the family of 
	$E_0\subseteq E(G)$ 
	such that 
	$V(G)\setminus V(E_0)$
	is independent in $G$.
	Clearly, 
	$E_0\in \Theta(G)$ 
	for every maximal matching $E_0$ of $G$. 
	Let $\tau(G)$ denote the minimum value of $|E_0|$ 
	over all $E_0\in \Theta(G)$.
	For example, if $G$ is the graph in 
	Figure~\ref{MG}, then $\tau(G)=1$
	and $\iota(\mid(G))=1$.
	
	In the section, we shall show that  	$\iota(\mid(G))=\tau(G)$ holds for any graph $G$.

\begin{Lem}\label{le3-1}
For any graph $G=(V,E)$, 
there exists $E_0\subseteq E(G)$ 
such that $\setm(E_0)$ is 
an isolating set of $\mid(G)$ and 
$|\setm(E_0)|=\iota(\mid(G))$. 
\end{Lem}

\begin{proof}
		\setcounter {clm}{0}
Let $S$ be a minimum isolating set of $\mid(G)$. Then $|S|=\iota(\mid(G))$. 
We further assume that 
$|S\cap V|$ is as small as possible. 
We shall show that $S\cap V=\emptyset$.

\Clm{cl0}
{
	If $m_{i,j}\in S$, then 
	$V(\mid(G))\setminus N_{\mid(G)}[S']
	=V(\mid(G))\setminus N_{\mid(G)}[S]$, 
	where 
	$S'=S\setminus \{v_i,v_j\}$.
}

Assume that $m_{i,j}\in S$. 
Observe that 
$N_{\mid(G)}[v_i]
\cup N_{\mid(G)}[v_j]
\subseteq 
N_{\mid(G)}[m_{i,j}]$,
implying that 
$V(\mid(G))\setminus N_{\mid(G)}[S']
=V(\mid(G))\setminus N_{\mid(G)}[S]$.
The claim holds. 

\Clm{cl1}
{
If $m_{i,j}\in S$, then 
$\{v_i,v_j\}\cap S=\emptyset$. 
}

Assume that $m_{i,j}\in S$. 
By Claim 1, $S'=S\setminus \{v_i,v_j\}$
is also an isolating of 
$\mid(G)$. 
Since $S$ is a minimum isolating set of $\mid(G)$, 
Claim~\ref{cl1} holds.

\Clm{cl2}
{
	For any $v_i\in S$, $v_i$ is not 
	an isolated vertex in $G$.
}

If $v_i$ is an isolated vertex in $G$,
then $S\setminus \{v_i\}$ is also an isolating set of $\mid(G)$,
contradicting the minimality of $S$. 
Claim~\ref{cl2} holds.

\Clm{cl3}
{$S\cap V(G)=\emptyset$. 
}

Suppose Claim~\ref{cl3} fails
and $v_i\in S$.
By Claim~\ref{cl2}, $N_G(v_i)\ne \emptyset$. 
Let $v_j$ be a vertex in $N_G(v_i)$. 
Without loss of generality, 
assume that $i<j$. 
By Claim~\ref{cl1}, $m_{i,j}\notin S$.
Now let 
$S_0=(S\setminus \{v_i\})\cup 
\{m_{i,j}\}$.
By Claim~\ref{cl0}, 
$V(\mid(G))\setminus N_{\mid(G)}[S_0]
=V(\mid(G))\setminus N_{\mid(G)}[S]$,
implying that $S_0$ is an isolating set of 
$\mid(G)$.
But, $|S_0|=|S|$ and 
$|S_0\cap V|<|S\cap V|$,
contradicting to the assumption on $S$.
Thus, Claim~\ref{cl3} holds.

By Claim~\ref{cl3}, $S\subseteq \setm(G)$. 
Now let $E_0$ be the set of edges $v_iv_j\in E(G)$, where $1\le i<j\le n$,
such that $m_{i,j}\in S$. 
Obviously, $\setm(E_0)=S$
and $|\setm(E_0)|=|S|=\iota(G)$. 
Hence the result holds.
\end{proof}
	
	\begin{Lem}\label{le3-2}
		For any subset $E_0$ of $E(G)$,
		$\setm(E_0)$ is an isolating set of 
		$\mid(G)$ 
		if and only if 
		$V(G)\setminus V(E_0)$ 
		is independent in  $G$.
	\end{Lem}
	
	\begin{proof}
		\setcounter {clm}{0}
		For any $E_0\subseteq E(G)$, 
		by definition, 
		$\mid(G)-N_{\mid(G)}[\setm(E_0)]$ 
		is exactly the graph $\mid(G-V(E_0))$.
		It can be  shown easily that 
		the following statements are equivalent, with an illustration in Figure~\ref{f6}.
		
		\begin{enumerate}
			\renewcommand{\theenumi}{\rm (\roman{enumi})}
			
			\item $\setm(E_0)$ is an isolating set of $\mid(G)$; 
			\item $\mid(G-V(E_0))$ has no edges;
			\item $V(G)\setminus 
			V(E_0)$ is independent in $G$.
		\end{enumerate}
Hence the lemma holds.
\end{proof}	
\begin{figure}[htbp]
	\centering
	\includegraphics[width=13cm]{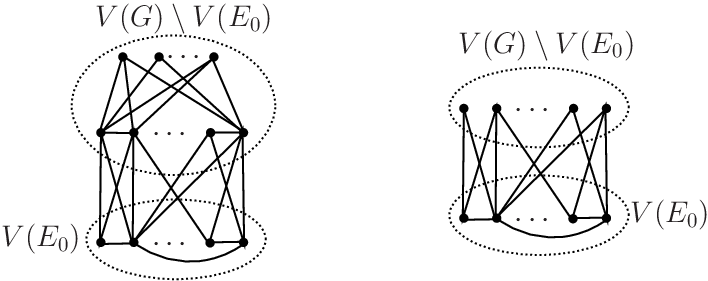}
	
{}\hspace{2 cm}	(a) \hspace{4 cm} (b) $V\setminus V(E_0)$ is independent 
	
	\caption{$E_0\subseteq E(G)$ } 
	\label{f6}
\end{figure}

By Lemmas~\ref{le3-1} and~\ref{le3-2}, we obtain the following conclusion 
on $\iota(\mid(G))$. 
	
	\begin{pro}\label{pro2-1}
		For any graph $G$, 
		$\iota (\mid(G))=\tau(G)$.
	\end{pro} 


Let $\nu(G)$ be the 
matching number of $G$, and 
let $\nu'(G)$ denote the minimum 
value of $|M|$ 
over all maximal matchings
$M$ of $G$. 
Obviously, $\nu'(G)\le \nu(G)$,
where the equality holds 
whenever $G$ is equimatchable. 

By definition, $M\in \Theta(G)$ for every 
maximal matching $M$ of $G$.
In the following, we show that 
$\tau(G)$ can be determined by 
considering the smallest maximal matchings of $G$, and thus Theorem~\ref{main1} follows
from Proposition~\ref{pro2-1}.

\begin{Lem}\label{le3-0}
	For any graph $G$, 
	there exists a maximal matching 
	$E_0$ of $G$ 
	such that  $|E_0|=\tau(G)$. 
\end{Lem}

\begin{proof}
	Let $E_0\in \Theta(G)$ with $|E_0|=\tau(G)$ such that 
	$|V(E_0)|$ is as large as possible. 
	We shall show that $E_0$ is a matching of $G$. 
	Suppose that $E_0$ is not
	a matching of $G$. Then 
	$E_0$ contains two edges 
	$e$ and $e'$ which have a common 
	end, say 
	$e=v_1v_2$ and $e'=v_1v_3$. 
	
	\noindent {\bf Claim A}:
		$N_G(v_3)\subseteq  V(E_0)$.
	
	Suppose that $N_G(v_3)\not \subseteq  V(E_0)$.
	Then, $e''=v_3v_t\in E(G)$ 
	for some vertex $v_t\in V(G)\setminus V(E_0)$. 
	Let $E'_0=(E_0\setminus \{e'\})\cup \{e''\}$.
	It is clear that $V(G)\setminus V(E'_0)\subseteq V(G)\setminus V(E_0)$,
	implying that $E'_0\in \Theta$.
	Obviously, $|E'_0|=|E_0|$
	and $|V(E'_0)|>|V(E_0)|$,
	contradicting the assumption of $E_0$.
	Thus, Claim A holds.
	
	By Claim A,
	$\{v_3\}\cup (V(G)\setminus V(E_0))$
	is independent in $G$,
	implying that 
	Then
	$
	V(G)\setminus V(E_0'')
	$
is independent in $G$,
where 	$E_0''=E_0\setminus \{e'\}$. 
It follows that $E''_0\in \Theta(G)$ and 
	$\tau(G)\le |E''_0|=|E_0|-1$,
	contradicting the assumption of $E_0$.
	
	Hence $E_0$ is a matching of $G$.
	Since $V(G)\setminus V(E_0)$ is 
	independent in $G$, $E_0$ is a maximal 
	matching of $G$.
	The conclusion holds.
\end{proof}

Since every maximal matching of $G$ 
belongs to $\Theta(G)$, 
by Lemma~\ref{le3-0}, 
we have the following relation between $\tau(G)$ and $\nu'(G)$. 

\begin{pro}\label{pro3-3}
	$\tau(G)=\nu'(G)$.
\end{pro}

Theorem~\ref{main1} follows directly from Propositions~\ref{pro2-1} and
\ref{pro3-3}. 

\section{Bounds for $\iota(\mid(G))$
\label{sec3}
} 

In this section, we shall find upper bounds and lower bounds 
for $\tau(G)$, and apply these results and Theorem~\ref{main1} 
to find the value of $\iota(\mid(G))$ for some families of graphs.


\begin{pro}\label{pro3-2}
	If $G$ has no isolated vertices, then 
		$\iota(\mid(G))\ge \frac{\gamma(G)}{2}$.
\end{pro}

\begin{proof}
By Theorem~\ref{main1},
there exists a maximal matching $M_0$ of $G$ with  $|M_0|=\tau(G)$.  
Then, $V(G)\setminus V(M_0)$ is independent in $G$. 
If $G$ has no isolated vertices, 
then each vertex in $V(G)\setminus V(M_0)$ 
is adjacent to some vertex in $V(M_0)$.
It follows that $V(M_0)$ is a dominating set of $G$, implying that $\gamma(G)\le |V(M_0)|$.
Since $|V(E_0)|=2|E_0|$,  we have 
$$
\tau(G)=|E_0|=\frac 12 |V(M_0)|\ge \frac 12\gamma(G).
$$
\end{proof}

There are graphs $G$ with 
$\iota(\mid(G))= \frac{\gamma(G)}{2}$.
For any graph $H$ with 
vertex set $\{w_1,w_2,\cdots,w_{2k}\}$ and 
a perfect matching $M$, 
if $G$ is obtained from 
$H$ 
by adding $2k$ new vertices $u_1,\dots, u_{2k}$ 
and adding new edges $w_iu_i$
for all $i=1,2,\dots,k$,
then $\iota(\mid(G))= k$ and $\gamma(G)=2k$.

\begin{pro}\label{pro3-0}
For any graph $G=(V,E)$, 
	$\iota(\mid(G))\ge \iota(G)$.
\end{pro} 

\begin{proof} 
By Theorem~\ref{main1},
there exists a maximal matching $M_0$ of $G$ with  $|M_0|=\iota(\mid(G))$. 
Let $V_0$ be a subset of $V(M_0)$
such that $|V_0|=|M_0|$ and each edge 
in $M_0$ is incident with one vertex in $V_0$. 
As $V\setminus V(M_0)$ is independent in $G$
and $V(M_0)\subseteq N_G[V_0]$, 
$V\setminus N_G[V_0]$ is independent in $G$,
implying that $V_0$ is isolating in $G$.
Thus, $\iota(G)\le |V_0|=|M_0|=\iota(\mid(G))$.
\end{proof} 

	
	Proposition~\ref{pro3-0} tells that 
	$\iota(G)$ is a lower bound of $\iota(\mid(G))$.
	The next result gives an upper bound 
	of $\iota(\mid(G))$ 
	in terms of $\iota(G)$. 
	
	\begin{pro}\label{pro3-5}
		Let $G$ be a graph with the maximum degree $\Delta$. Then $\iota(\mid(G))\le \Delta\iota(G)$. Moreover, the bound is sharp.
	\end{pro}
	
	\begin{proof}
		Let $S$ be a minimum isolating set of $G$ and $Q=V(G)\backslash N_G[S]$.
		Let $E_0$ be the set of edges in $G$ 
		each of which is incident with some vertex in $S$. 
		Obviously, 
		$N_G[S]\subseteq V(E_0)$.
		Thus, $V(G)\setminus V(E_0)\subseteq V(G)\setminus N_G[S]$ is independent in $G$.
		It follows that $E_0\in \Theta(G)$.
		
		By the assumption of $E_0$, 
		$|E_0|\le \Delta |S|
		=\Delta \iota(G)$.
		It follows that 
		$$
		\iota(\mid(G))=\tau(G)\le |E_0|\le \Delta \iota(G).
		$$
		
		The sharpness can be seen 
		when $G$ is the complete bipartite graph $K_{k,k}$ for any $k\ge 1$,
		as $\iota(K_{k,k})=1$
		and $\iota(\mid(K_{k,k}))=k=\Delta$. 
	\end{proof}
	

Another lower bound of $\tau(G)$ is in terms of the size of $G$ and the maximum degree.

\begin{pro}\label{pro3-20} 
Let $G$ be a graph with the maximum degree $\Delta$.
Then $\iota(\mid(G))\ge \frac{|E(G)|}{2\Delta-1}.$

\end{pro}

\begin{proof}
	By Proposition~\ref{pro2-1},
	there exists a maximal matching $M_0$ of $G$ with  $|M_0|=\nu'(G)$. 
	By definition, $V(G)\setminus V(M_0)$ 
	is independent in $G$,
	implying that 
	every edge in $G$ is incident with some vertex in $V(M_0)$.
	Thus,
	$$
	|E(G)|\le 
	\sum_{v_i\in V(E_0)}d(v_i)	-|M_0|
	\le |V(M_0)|\Delta-|M_0|
	=2|M_0|\Delta-|M_0|
	=|M_0|(2\Delta-1),
	$$
	which implies $|M_0|\ge \frac{|E(G)|}{2\Delta-1}$.
	Since $|E_0|=\nu'(G)$,
	the result follows from Theorem~\ref{main1}. 
\end{proof}

For any maximal independent set $V_0$ of $G$,
there exists a subset $E_0$ of $E(G)$ 
with  $|E_0|\le n-|V_0|$
such that 
every vertex in $V(G)\setminus V_0$ is incident with some edges in $E_0$.
Thus, 
$V(G)\setminus N_G[V(E_0)]$
is independent in $G$ and
$E_0\in \Theta(G)$.
It follows that 
$$
\tau(G)\le |E_0|\le n-|V_0|.
$$
Therefore the following upper 
bound for $\iota(\mid(G))$ is obtained.

 \begin{pro}\label{pro3-4}
	$
	\iota(\mid(G))=\tau(G)\le n-\alpha(G),
	$
where $\alpha(G)$ is the independent number of $G$.
\end{pro}

As a special case, we are now going to 
apply the above results to determine 
$\iota(\mid(G))$ 
when $G$ is a path  or a cycle.

\begin{thm}\label{th3-3}
	Let $P_n$ be a path of order $n$, where $n\ge2$. Then
	$\iota (\mid(P_n))=\lfloor\frac{n+1}{3}\rfloor$.
\end{thm}

\begin{proof}
By Proposition~\ref{pro3-20}, 
$\iota(\mid(P_n))\ge \frac{n-1}{3}\ge\lfloor \frac{n+1}{3}\rfloor.$
By Proposition~\ref{pro2-1}, 
it remains to show that 
$\tau(P_n)\le\lfloor \frac{n+1}{3}\rfloor.$

Let $v_1v_2\dots v_n$ denote the path $P_n$. 
If $n\equiv 0\pmod{3}$, then 
$
M_1=\{v_{3i-1}v_{3i}: i=1,\dots, \frac{n}{3}\}
$
is a maximal matching of $P_n$.
By Proposition~\ref{pro3-3}, 
$\tau(P_n)=\nu'(P_n)\le |M_1|=\frac{n}{3}
=\lfloor \frac{n+1}{3}\rfloor $.

If $n\equiv 1\pmod{3}$, then
$
M_2=\{v_{3i-1}v_{3i}: i=1,\dots, \frac{n-1}{3}\}
$
is a maximal matching of $P_n$.
By Proposition~\ref{pro3-3}, 
$\tau(P_n)= \nu'(P_n)\le |M_2|=\frac{n-1}{3}=\lfloor \frac{n+1}{3}\rfloor $.

If $n\equiv 2\pmod{3}$, then 
$
M_3=\{v_{3i-1}v_{3i}: i=1,\dots, \frac{n-2}{3}\}\cup \{v_{n-1}v_n\}$
is a maximal matching of $P_n$.
By Proposition~\ref{pro3-3}, 
$\tau(P_n)=\nu'(P_n)\le |M_3|=\frac{n+1}{3}=\lfloor \frac{n+1}{3}\rfloor $.

Thus, the theorem holds.
\end{proof}

\begin{thm}\label{th3-4}
Let $C_n$ be a cycle of order $n$,
where $n\ge3$. Then
$\iota (\mid(C_n))=\lfloor \frac{n+2}{3}\rfloor$.
	\end{thm}

\begin{proof}
By Proposition~\ref{pro3-20}, 
$\iota(\mid(G))\ge \frac{n}{3}$.
As $\iota(\mid(G))$ is an integer, 
$\iota(\mid(G))\ge \lfloor \frac{n+2}{3}\rfloor$.
By Proposition~\ref{pro2-1}, 
it remains to show that 
$\tau (C_n)\le \lfloor \frac{n+2}{3}\rfloor$.
Let $v_1v_2\dots v_nv_1$ denote the cycle $C_n$. 
	
If $n\equiv 0\pmod{3}$, then the set
$M_1=\{v_{3i-1}v_{3i}: i=1,\dots, \frac{n}{3}\}$
is a maximal matching of $C_n$.
    By Proposition~\ref{pro3-3}, 
	$\tau(C_n)\le \nu'(C_n)= |M_1|=\frac{n}{3}
	=\lfloor \frac{n+2}{3}\rfloor$.
	
	If $n\equiv 1\pmod{3}$, then 
	$M_2=\{v_{3i-1}v_{3i}: i=1,\dots, \frac{n-1}{3}\}\cup \{v_nv_1\}$
	is a maximal matching of $C_n$.
	By Proposition~\ref{pro3-3}, 
	$\tau(C_n)\le \nu'(C_n)= |M_2|=\frac{n+2}{3}=\lfloor \frac{n+2}{3}\rfloor$.
	
	If $n\equiv 2\pmod{3}$, then 
	$M_3=\{v_{3i-1}v_{3i}: i=1,\dots, \frac{n-2}{3}\}\cup \{v_{n-1}v_n\}$
	is a maximal matching of $C_n$.
	By Proposition~\ref{pro3-3}, 
	$\tau(C_n)= \nu'(C_n)\le  |M_3|=\frac{n+1}{3}=\lfloor \frac{n+2}{3}\rfloor$.
	
	Since $\tau (C_n)\ge \lfloor \frac{n+2}{3}\rfloor$,
	by the above arguments, 
	the theorem holds.
\end{proof}

\section{Maximum value of $\iota(\mid(G))$
over all graphs $G$ of order $n$
}

Given any connected graph $G$ of order $n$, what is the maximum value of $\iota(\mid(G))$?
In this section, we determine a sharp upper bound of $\iota(\mid(G))$ in terms of $n$.

\begin{Lem}\label{le3-5}
	$\iota(\mid(K_n))=\lfloor \frac{n}{2}\rfloor$,
	and $\iota(\mid(K_{n_1,n_2}))
	=\min\{n_1,n_2\}$.
\end{Lem} 

\begin{proof}
	Clearly, $K_n$ has a maximal matching of size $\lfloor \frac{n}{2}\rfloor$. 
	Thus, $\tau(K_n)\le \lfloor \frac{n}{2}\rfloor$
	by Theorem~\ref{main1}.
	
	Let $E_0\in \Theta(K_n)$
	with $|E_0|=\tau(K_n)$. 
	Then, $V(K_n)\setminus V(E_0)$ 
	is an independent set of $K_n$,
	implying that 
	$|V(K_n)\setminus V(E_0)|\le 1$.
	It follows that $|V(E_0)|\ge n-1$. 
	Since $|V(E_0)|\le 2|E_0|$, we have 
	$$
	\tau(K_n)=|E_0|\ge \frac{1}{2} |V(E_0)|\ge \frac{1}{2} (n-1).
	$$
	Since $	\tau(K_n)$ is an integer, 
	$\tau(K_n)\ge \lfloor \frac{n}{2}\rfloor$.
	By Proposition~\ref{pro2-1}, 
	 $\iota(\mid(K_n))\ge \lfloor \frac{n}{2}\rfloor$.
	 Hence 	 $\iota(\mid(K_n))= \lfloor \frac{n}{2}\rfloor$.

	Assume that $n_1\le n_2$
	and $G=K_{n_1,n_2}$
	with a bipartition $(A,B)$, 
	where $|A|=n_1$ and $|B|=n_2$.
	Then, $G$ contains an edge set $E_0$ of size $n_1$
	such that $V(G)\setminus V(E_0)$  
	is independent in $G$,
	implying that $\tau(G)\le n_1$.
	
	Let $E_0$ be any set in $\Theta(G)$
	with $|E_0|=\tau(G)$. 
	Then, $V(G)\setminus V(E_0)$ 
	is an independent set of $G$,
	implying that either 
	$|A\setminus V(E_0)|=0$
	or 	$|B\setminus V(E_0)|=0$.
	It follows that $|E_0|\ge n_1$. 
	Thus, $\tau(G)\ge n_1$.
	
	Hence $\tau(K_{n_1,n_2})= n_1$, by Theorem~\ref{main1} , 
	$\iota(\mid(K_{n_1,n_2}))=n_1$.
\end{proof}

A graph is said to be 
{\it randomly matchable} if every matching of $G$ can be extended to a perfect matching. 
It was shown by Summer \cite{Sum}
that the connected randomly matchable graphs are precisely 
$K_{2k}$ and $K_{k,k}$ for $(k\ge  1)$.
A {\it near-perfect matching} of a graph $G$ 
is a matching in which a single vertex in $G$ is left unmatched. 

\begin{thm}\label{th3-6}
For any connected 
	graph $G$ of order $n$, 
	$
	\iota(\mid(G))\le \lfloor \frac{n}{2}\rfloor,
	$
	where the equality holds \iff 
	\begin{enumerate}
		\item when $n$ is even, 
		either $G\cong K_{n}$ 
	or $G\cong K_{n/2,n/2}$; and 
	    \item when $n$ is odd,  
	    every maximal matching in $G$ 
	is near-perfect. 
		\end{enumerate}
\end{thm} 

\begin{proof}
Clearly, $\nu'(G)\le \frac{n}{2}$
and thus $\nu'(G)\le \Floor{\frac{n}{2}}$. 
Then, by Theorem~\ref{main1}, 
$\tau(G)\le \Floor{\frac{n}{2}}$.

\noindent {\bf Case 1}: $n$ is even,
say $n=2k$. 

If $G$ is neither $K_{2k}$ nor $K_{k,k}$, 
then, 
by a result due to Summer~\cite{Sum}, 
$G$ is not randomly matchable.
It follows that $G$ has a matching $M_0$ such that $M_0$ cannot be extended to a perfect matching of $G$, 
implying that $G$ has a maximal matching $M$ which is not perfect.
Thus, 
$\nu'(G)\le |M|<k$.

If $G$ is either $K_{2k}$ or $K_{k,k}$,
then $\tau(G)=k$ by Lemma~\ref{le3-5}. 
Hence the result holds when $n$ is even.

\noindent {\bf Case 2}: $n$ is odd,
say $n=2k+1$. 

If $G$ has a maximal matching $M$ 
which is not near-perfect, then 
$\nu'(G)\le |M|\le \frac{|V(G)|-2}{2} < k$.
By Theorem~\ref{main1}, 
$\iota(\mid(G))<k$.

Now assume that every maximal matching $M$ of $G$ is near-perfect,
i.e., $|M|=k$.
It follows that $\nu'(G)=k$
by the definition of $\nu'(G)$. 
By Theorem~\ref{main1}, 
$\iota(\mid(G))=k$. 

Thus the result holds when $n$ is odd.
\end{proof} 

\section{Upper bound of $\iota(\mid(T))$ for trees $T$
} 

In this section, we study 
$\iota(\mid(T))$ for a tree $T$. 
Due to Caro and Hansberg \cite{Caroa},  
$\iota (T)\le \frac{n}{3}$ holds 
for any tree $T$ of order $n\ge 3$. 
In this section, we shall 
prove Theorem~\ref{main3} and then 
determine all trees $T$ with 
$\iota(\mid(T))= \lfloor 
\frac{n-1}{2}\rfloor$.

\vspace{0.5 cm}
\noindent {\it Proof} of Theorem~\ref{main3}: 
By Theorem~\ref{th3-6}, 
$\iota (\mid(T))\le \frac{n-1}{2}$
if $n$ is odd;
$\iota (\mid(T))\le \frac{n-2}{2}$
otherwise. 
Thus, 
$\iota (\mid(T))\le \Floor{\frac{n-1}{2}}$.

Now we are going to show that the result is sharp. 
For any $n=2k+i$, where $i=1,2$, 
the tree $T_i$ 
of order $n$ shown in Figure~\ref{f7} 
has the property that 
$\nu'(T_i)=k=\Floor{\frac{n-1}2}$,
implying that 
$\iota(\mid(T_i))=\Floor{\frac{n-1}2}$.
Thus, the result holds.
\proofend 

\begin{figure}[H]
	\centering
	\includegraphics[width=10cm]{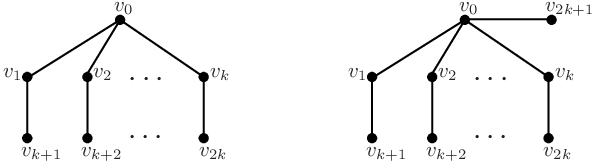}

~~~~~~~(a) $T_1$\hspace{4.2 cm}~~~~~(b) $T_2$ ~~~~
	
\caption{
$\iota(\mid(T_i))=k=
\Floor{\frac{|V(T_i)|-1}{2}}$
for both $i=1,2$, where $k\ge 2$
}

	\label{f7}
\end{figure}

Now we are going to determine
all trees $T$ such that 
$\iota (\mid(T))=\Floor {\frac{|V(T)|-1}{2}}$.

Let $o(G)$ be the number of odd 
components of $G$, 
$d_G(u,v)$ be the distance of  vertices $u$ and $v$ in $G$, 
and 
$Diam(G)$ be the maximum value of 
$d_G(u,v)$ over all pairs of vertices 
$u$ and $v$ in $G$.

 \begin{Lem}\label{le5-0}
 	Let $T$ be a tree of order $n$, where $n\ge 3$, such that  
 	$\iota (\mid(T))=\Floor {\frac{n-1}{2}}$.
 	\begin{enumerate}
 		\item  	\label{no1}
 		 For matching $N$ of $T$, 
 		$o(T-V(N))\le 2$; 
 		
 		\item 	\label{no2}
 		 		For matching $N$ of $T$, 
 		 if $o(T-V(N))\ge 1$, 
 		then $o(T-V(N))= 2$ when $n$ is even, and $o(T-V(N))= 1$ otherwise.
 		Furthermore,  each even component of $T-V(N)$ 
 			is isomorphic to $K_2$;  
 		\item 	\label{no3}
 		if $n$ is odd and 
 		$n\ge 5$, then $d_T(u,v)=4$ for every pair of leaves $u$ and $v$ in $T$; 
 		
 		\item \label{no3-1}
 		if $n$ is even and $n\ge 6$, 
 		then $Diam(T)\ge 4$; 
 		
 		\item 	\label{no4}
 		if $n$ is even, then $Diam(T)\le 7$; and
 		
 		\item 	\label{no5}
 		if $n$ is even and 
 		$Diam(T)\ge 5$, then
 		$d_T(u,v)\ge 4$ for any two 
 		leaves $u$ and $v$ in $T$. 
 	\end{enumerate}
 \end{Lem}

\begin{proof} 
By Theorem~\ref{main1}, 
$\iota (\mid(T))=\Floor {\frac{n-1}{2}}$
implies that $\nu'(T)=\Floor {\frac{n-1}{2}}$.
Let $N$ be any matching of $T$.

	\ref{no1}.
For any maximal matching $N_0$ of $T-V(N)$, 
we have 
$$
|N_0|\le \frac{|V(T)|-|V(N)|-o(T-V(N))}{2}.
$$
Observe that if $N_0$ is a maximal matching $N_0$ of $T-V(N)$, 
then $N_0\cup N$ is a maximal matching of $T$. Thus,
$$
\nu'(T)\le |N|+
\frac{|V(T)|-|V(N)|-o(T-V(N))}{2}
=\frac{n-o(T-V(N))}{2}.
$$
Since $\nu'(T)=\Floor {\frac{n-1}{2}}$,
the result follows.

\ref{no2}.
Assume that $i=o(T-V(N))\ge 1$.
Since $|V(N)|$ is even, 
$i$ is even \iff $n$ is even.
Thus,  $i\ge 2$ if $n$ is even
and $i\ge 1$ otherwise. 
But, by the result in \ref{no1}, 
$i= 2$ if $n$ is even
and $i=1$ otherwise. 

Suppose some even component 
$T_0$ of $T-V(N)$ with $T_0\not\cong K_2$.
Then $T_0$ has a matching $N_0$ such that $T_0-V(N_0)$ 
has an isolated vertex.
It follows that 
$N\cup N_0$ is a matching of $T$ such that 
$o(T-V(N\cup N_0))\ge 3$, 
a contradiction to \ref{no1}.
Hence \ref{no2} holds.

\ref{no3}.  
Let $v_0v_1\cdots v_k$ 
be any path in $T$, where $v_0$ and $v_k$ are leaves in $T$.
In order to prove that $k=4$, 
it suffices to show that when $k\ne 4$, 
$T$ has a matching $N$ 
such that both $v_0$ and $v_k$ are isolated 
vertices in $T-V(N)$,
implying that $o(T-V(N))\ge 2$.
However, as $n$ is odd, 
$o(T-V(N))$ is odd and 
thus $o(T-V(N))\ge 3$, 
a contradiction to 
the result of \ref{no1}. 

If $k\ge 5$, then 
$N=\{v_1v_2, v_{k-2}v_{k-1}\}$
is a matching of $T$ such that 
$v_0$ and $v_k$ are isolated 
vertices in $T-V(N)$.
If $k=3$, then $N=\{v_1v_2\}$ is  
a matching of $T$ such that 
$v_0$ and $v_3$ are isolated 
vertices in $T-V(N)$.
If $k=2$, then $N=\{v_1w\}$ is  
a matching of $T$ such that 
$v_0$ and $v_2$ are isolated 
vertices in $T-V(N)$, 
where $w$ is a vertex in $N_T(v_1)\setminus \{v_0,v_2\}$.
Since $n\ge 5$, such a vertex $w$ in 
$N_T(v_1)\setminus \{v_0,v_2\}$ exists. 

Thus \ref{no3} holds.

\ref{no3-1}. Clearly, $Diam(T)\ne 2$.
Otherwise, $T$ is a star and 
$\nu'(T)=1<\Floor{\frac {n-1}2}$,
a contradiction. 

Now suppose that $Diam(T)=3$.
Let $P:=u_0u_1u_2u_3$ be a path in $T$,
where both $u_0$ and $u_3$ are leaves in $T$.
For the matching $N=\{u_1u_2\}$, 
both $u_0$ and $u_3$ are isolated vertices in $T-V(N)$.
Let $T_1, T_2, \cdots, T_s$
be the components of  $T-V(P)$. 
By the result in \ref{no2}, $T_i$
is isomorphic to $K_2$
for each $i=1,2,\dots,s$.
Since $n\ge 6$, we have $s\ge 1$.
Clearly, $T$ can be obtained from 
path $P$ and subtrees $T_1, T_2,\dots, T_s$ by adding edges $e_1,e_2,\dots, e_s$, where each $e_i$ 
joins a vertex in $\{u_1,u_2\}$ 
to a vertex in $V(T_i)$. 
It is clear that $Diam(T)\ge 4$, 
a contradiction to the assumption
that $Diam(T)=3$.

Thus, \ref{no3-1} holds. 

\ref{no4}. Let $n$ be even and 
$u_0u_1\cdots u_t$ be a longest path in $T$.
Then both 
$u_0$ and $u_t$ are leaves in $T$.
Suppose that  $t=Diam(T)\ge 8$.
Observe that $N=\{u_1u_2,u_{k-2}u_{k-1}\}$ 
is a matching of $T$ 
such that both $u_0$ and $u_k$  
are isolated vertices in $T-V(N)$,
implying that $o(T-V(N))\ge 2$.
By \ref{no2}, each even component 
of $T-V(N)$ is isomorphic to $K_2$.
However, the component of $T-V(N)$
containing vertices $u_3,u_4,\dots, 
u_{k-3}$ has at least $k-5\ge 3$ vertices,
a contradiction.
 
Hence \ref{no4} holds.

\ref{no5}. Assume that $n$ is even 
and $Diam(T)\ge 5$.

We first show that $T$ has  no 
path $u_0u_1u_2u_3$ of length $3$
connecting two leaves of $T$.
Otherwise, $N=\{u_1u_2\}$ 
is a matching of $T$ such that 
both $u_0$ and $u_3$ are isolated vertices of $T-V(N)$,
forming two odd components
$T_1$ and $T_2$.
Let $T_3,T_4,\dots, T_r$ 
be the other components of $T-V(N)$.
By \ref{no1} and \ref{no2},  $T_i\cong K_2$
for each $i=3,4,\dots,r$.
Then, $T$ is a tree obtained 
from the path $u_0u_1u_2u_3$ 
and the 
subtrees $T_3,T_4,\cdots, T_r$ 
by adding one edge $e_i$ joining 
one vertex  in $\{u_1,u_2\}$ 
to one vertex in $V(T_i)=\{x_i,y_i\}$,
say $x_i$, 
for each $i\in \{3,4,\dots,r\}$.
Since $Diam(T)\ge 5$, 
there must be two edges in $\{e_i: 3\le i\le r\}$, 
say $e_3$ and $e_4$, 
such that $e_3=u_1x_3$
and $e_4=u_2x_4$, as shown in
Figure~\ref{f8}. 
However,  $N'=\{e_3,e_4\}$ must be a matching 
of $T$ such that $u_0,u_3, y_3$ and $y_4$ are isolated vertices in $T-V(N')$,
a contradiction to \ref{no1}.

\begin{figure}[H]
	\centering
	\includegraphics[width=3.5cm]{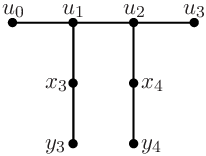}
	
	
	\caption{A subtree of $T$
	}
	
	\label{f8}
\end{figure}

Now we are going to 
show that $T$ has  no 
path $u_0u_1u_2$ of length $2$
connecting two leaves of $T$.
Otherwise, $u_1$ is adjacent to 
some vertex $w\in V(T)\setminus \{u_0,u_1,u_2\}$ as $n\ge \Delta(T)+1\ge 6$.
Note that $N=\{u_1w\}$ is a matching 
of $T$ such that 
both $u_0$ and $u_2$ are 
isolated vertices 
in $T-V(N)$.
Let $T_3,\cdots,T_r$ be the components 
of $T-V(N)$ such that 
$T_i\cap \{u_0,u_2\}=\emptyset$
for each $i=3,4,\dots,r$.
By (i) and (ii), $T_i\cong K_2$
for each $i\in \{3,4,\dots,r\}$,
and $T$ is obtained from 
the subtree $T_0$, 
induced by $\{w\}\cup 
\{u_i: i=0,1,2\}$,
and subtrees $T_3,\cdots,T_r$
by adding an edge $e_i$ 
joining a vertex in $\{u_1,w\}$ 
to a vertex in $V(T_i)=\{x_i,y_i\}$,
say $x_i$, 
for each $i\in \{3,4,\dots,r\}$.

By the conclusion in the previous paragraph, 
$d_T(u,v)\ne 3$ for any two leaves 
$u$ and $v$ in $T$, 
 implying that $e_i=wx_i$
for each $i\in \{3,4,\dots,r\}$.
It follows that $Diam(T)\le 4$,
a contradiction
to the given condition that 
$Diam(T)\ge 5$. 

Hence \ref{no5} holds.
\end{proof} 

In the following, 
we are going to determining all 
trees 
$T$ with 
$\iota (\mid(T))=\Floor {\frac{|V(T)|-1}{2}}$.
Obviously, $|V(T)|\ge 3$ for such trees.

\begin{thm}\label{th3-8}
For any tree $T$ of order $n$, 
where $n\ge 3$,   if 
$\iota (\mid(T))=\Floor {\frac{n-1}{2}}$,
then $T$ is a tree described below:
\begin{enumerate}
	\item when $3\le n\le 4$, 
	$T$ is $P_n$ or $K_{1,3}$;

	\item when $n\ge 5$
	and $n$ is odd, $T$ is the 
	tree $T_1$ shown in Figure~\ref{f7} (a) for some $k\ge 2$; 
	
	\item when $n\ge 6$,
	$n$ is even and 
	$Diam(T)\le 4$, 
	$T$ has a leaf $w$ such that 
	$T-w$ is isomorphic to the tree  
	$T_1$ shown in Figure~\ref{f7} (a);
	and 
	
	\item when $n\ge 6$, 
	$n$ is even and 
	$Diam(T)\ge 5$,
	$T$ is a tree shown in Figure~\ref{f5-3}.
\end{enumerate} 
\end{thm} 

\begin{figure}[!ht]
	\centering
	\includegraphics[width=15cm]
	{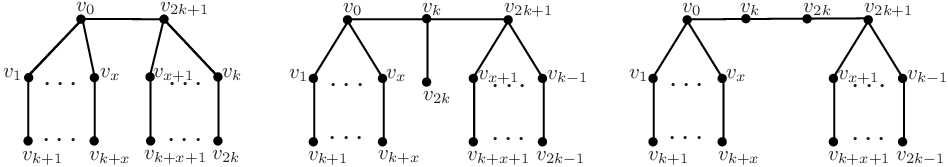}
	
	(a)\hspace{4 cm}(b)\hspace{4.5 cm}(c)~~~~
	
	\caption{Trees $T$ with 
		$\iota(\mid(T))=\frac{|V(T)|-2}{2}$}
	\label{f5-3}
\end{figure}

\begin{proof}
	(i). It can be verified directly. 
	
	(ii). Assume that $n\ge 5$ and 
	$n$ is odd. 
	If $n=5$, then $T\cong P_5$
	by Lemma~\ref{le5-0} \ref{no3}, 
	and the result holds.
	In the following, assume that $n\ge 7$. 
	
	By Lemma~\ref{le5-0} \ref{no3} again, 
	$d_T(u,v)=4$ for each pair of leaves 
	$u$ and $v$ in $T$, implying that
	$Diam(T)=4$.
	Let $u_1u_2u_3u_4u_5$ be a longest path in $T$, where 
 both $u_1$ and $u_5$ are leaves in $T$.
	By Lemma~\ref{le5-0} \ref{no3} again,
	we have  
	$N_T(u_2)=\{u_1,u_3\}$ and
	$N_T(u_4)=\{u_3,u_5\}$. 
	Thus, $d_T(u_3)\ge 3$. 
	
	Assume that  $N_T(u_3)\setminus \{u_2,u_4\}=\{v_1,v_2,\dots,v_r\}$,
	where $r=d_T(u_3)-2$.
	By Lemma~\ref{le5-0} \ref{no3} again, 
	for each $i\in [r]$, $d_T(v_i)\ge 2$,
	and $Diam(T)=4$ implies that  
	each vertex in 
	$N_T(v_i)\setminus \{u_3\}$ is a leaf 
	of $T$. 
	By Lemma~\ref{le5-0} (iii) \ref{no3}, 
	for each $i\in [r]$, $d_T(v_i)=2$.
	It follows that $T$ is isomorphic 
	to $T_1$ shown in Figure~\ref{f7} (a),
	where 
	$k=\frac{n-1}2$.
	
	Hence (ii) holds.
	
	(iii). Assume that $n\ge 6$, $n$ is even and $Diam(T)\le 4$.
	By Lemma~\ref{le3-0} \ref{no3-1}, 
	$Diam(T)=4$.
	
	We first prove the following claims.
	
\noindent {\bf Claim 1}: 
For any longest 
path $u_1u_2u_3u_4u_5$ in $T$,
all vertices in $(N_T(u_2)\cup N_T(u_4))
\setminus \{u_3\}$ are leaves of $T$.

The claim  follows directly 
from the condition that $Diam(T)=4$.	
	
\noindent {\bf Claim 2}: 
For any longest 
path $u_1u_2u_3u_4u_5$ in $T$,
$(N_T(u_2)\cup N_T(u_3)\cup N_T(u_4))\setminus \{u_1,u_5\}$ contains at most $1$ leaf in 
$T$.

If $T$ has at least two leaves 
contained in $(N_T(u_2)\cup N_T(u_3)\cup N_T(u_4))\setminus \{u_1,u_5\}$, 
then it can be verified that 
$T$ has a matching $N$ such that 
$T-V(N)$ contains at least three isolated vertices, contradicting 
Lemma~\ref{le5-0} \ref{no1}. 
	
Hence Claim 2 holds.
	
\noindent {\bf Claim 3}: 
$T$ has a longest 
path $u_1u_2u_3u_4u_5$ in $T$
such that 
$(N_T(u_2)\cup N_T(u_3)\cup N_T(u_4))\setminus \{u_1,u_5\}$ contains exactly  $1$ leaf of $T$.

Assume that Claim 3 fails. 
Let  $P:u_1u_2u_3u_4u_5$ be a longest path in $T$. 
Then,  $d_T(u_2)=d_T(u_4)=2$.
Since $n=|V(T)|$ is even, 
$T-V(P)$ has an odd component 
$T_0$. 
Thus, $T$ has an edge $e_0$ joining 
$u_3$ to some vertex $u_4'$ in $T_0$.
Since $Diam(T)=4$, $T_0$ must be a star with $u_4'$ as its center.
As Claim 3 fails, 
$|V(T_0)|\ge 3$. 
But, then there is a path of length $4$:
$u_1u_2u_3u_4'u_5'$, where 
$u_5'\in N_{T_0}(u'_4)$,
such that each vertex $v$ in
$V(T_0)\setminus \{u'_4,u'_5\}$ 
is a leaf of $T$ and $v\in N_T(u'_4)$,
implying that Claim 3 holds, a contradiction. 

Thus, Claim 3 holds.

\noindent {\bf Claim 4}: 
$T$ contains a leaf $w$ such that 
$T-w$ is isomorphic to $T_1$,
where $T_1$ is the graph 
in Figure~\ref{f7} (a) for some $k\ge 2$.

By Claim 3, 
$T$ has a longest 
path $u_1u_2u_3u_4u_5$ in $T$
such that 
$(N_T(u_2)\cup N_T(u_3)\cup N_T(u_4))\setminus \{u_1,u_5\}$ contains exactly  $1$ leaf of $T$,
say $w$.
We may assume that 
$w\in N_T(u_2)\cup N_T(u_3)$. 

Then, $N=\{u_2u_3\}$ is a matching 
of $T$ such that both $w$ and $u_1$ 
are isolated vertices of $T-V(N)$.
By Lemma~\ref{le5-0} (ii), 
for each component $T_i$ of 
$T-(V(N)\cup \{u_1,w\})$, 
$T_i$ is isomorphic to $K_2$
and 
$|N_T(u_3)\cap V(T_i)|=1$
for each component $T_i$ of 
$T-(V(N)\cup \{u_1,w\})$.
Therefore, $T-w$ is isomorphic to 
$T_1$,  
the tree in Figure~\ref{f7} 
(a) for some $k\ge 2$.

Hence (iii) holds.

(iv).  
Assume that $n\ge 6$, $n$ is even 
and $Diam(T)\ge 5$.
Let $u_0u_1u_2\cdots u_t$ be a longest path in $T$.
Thus, $t=Diam(T)$.
By Lemma~\ref{le5-0} \ref{no4}, 
$5\le t\le 7$.
Clearly, both $u_0$ and $u_t$ 
are leaves of $T$.
Note that 
$N=\{u_1u_2,u_{t-2}u_{t-1}\}$
is a matching of $T$ such that 
both $u_0$ and $u_t$ are isolated vertices 
in $T-V(N)$.
Let $T_3,T_4,\dots,T_r$ be the components of $T-(V(N)\cup \{u_0,u_t\})$.
By \ref{no1} and \ref{no2} in Lemma~\ref{le5-0}, $T_i\cong K_2$
for each $i\in \{3,4,\dots,r\}$.

\noindent {\bf Case 1}: $t=5$. 

In this case, 
$T$  is a tree obtained from the path 
$u_0u_1u_2\dots u_5$ and 
subtrees $T_3,\cdots,T_r$
by adding an edge $e_i$ 
joining a vertex in $V(T_i)=\{x_i,y_i\}$,
say $x_i$, to 
a vertex in $\{u_1,u_2,u_3,u_4\}$ 
for each $i\in \{3,4,\dots,r\}$.
By Lemma~\ref{le5-0} \ref{no5}, 
$d_T(v_1,v_2)\ge 4$
for any two leaves $v_1$ and $v_2$ in $T$.
Thus, for each $i\in \{3,4,\dots,r\}$,
$e_i$ must join $x_i$ to a vertex in $\{u_2,u_3\}$.
In this case, 
$T$ is the graph in Figure~\ref{f5-3} (a)
for some $k\ge 2$. 

\noindent {\bf Case 2}: $t=6$. 

In this case, some subtree $T_i$,
say $T_3$, contains vertex $u_3$.
Since $T_3\cong K_2$,  $u_3$ must be  adjacent to some leaf of $T$, say $v$,
as shown in Figure~\ref{f9} (a).

\begin{figure}[H]
	\centering
	\includegraphics[width=15cm]{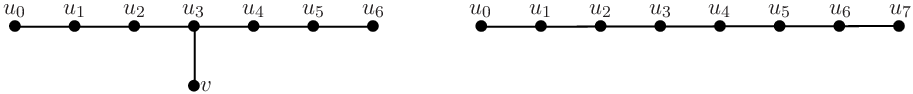}
	
	~~~(a) $t=6$\hspace{5.2 cm} (b) $t=7$ ~~~~
	
	\caption{Subtrees of $T$}
	
	\label{f9}
\end{figure}

Then,  $T$ is a tree obtained from 
the subtree $T_0$, 
induced by 
$\{u_i:0\le i\le 6\}\cup \{v\}$ 
(see Figure~\ref{f9} (a)), 
and 
subtrees $T_4,\cdots,T_r$
by adding an edge $e_i$ 
joining a vertex in $V(T_i)=\{x_i,y_i\}$,
say $x_i$, 
to 
a vertex in $\{u_1,u_2,u_4,u_5\}$ 
for each $i\in \{4,5,\dots,r\}$.
By Lemma~\ref{le5-0} \ref{no5}, 
$d_T(v_1,v_2)\ge 4$
for any two leaves $v_1$ and $v_2$ in $T$.
Thus, for each $i\in \{4,5,\dots,r\}$,
$e_i$ must join $x_i$ to a vertex in $\{u_2,u_4\}$.
In this case,
$T$ is the graph in Figure~\ref{f5-3} (b)
for some $k\ge 2$. 

\noindent {\bf Case 3}: $t=7$. 

In this case, some subtree $T_i$,
say $T_3$,  has its vertex set 
$\{u_3, u_4\}$, as shown in Figure~\ref{f9} (b).
Then, $T$  is a tree obtained from the path 
$u_0u_1u_2\dots u_7$ and 
subtrees $T_4,\cdots,T_r$
by adding an edge $e_i$ 
joining a vertex in $V(T_i)=\{x_i,y_i\}$,
say $x_i$, 
to 
a vertex in $\{u_1,u_2,u_5,u_6\}$ 
for each $i\in \{4,5,\dots,r\}$.
By Lemma~\ref{le5-0} \ref{no5}, 
$d_T(v_1,v_2)\ge 4$
for any two leaves $v_1$ and $v_2$ in $T$.
Thus, for each $i\in \{4,,5\dots,r\}$,
$e_i$ must join $x_i$ to a vertex in $\{u_2,u_5\}$.
In this case, 
$T$ is the graph in Figure~\ref{f5-3} (c)
for some $k\ge 2$. 

Hence (iv) holds.
\end{proof} 

\section{Conclusion} 

This paper mainly investigates the partial domination problem with no restriction on the properties $P$ or $Q$ but with $G[V\backslash N_G[S]]$ having the property that $G[V\backslash N_G[S]]$ is an independent set of middle graphs. There are many problems about this issue that deserve further to be studied. We propose the following problems. 

\begin{prob}\label{prob1}
What other transformation graphs in graph theory can be studied for their partial domination? What is the corresponding bound or value?
\end{prob}

\begin{prob}\label{prob2}
What is the relationship between the $\mathcal{F}$-isolation number of different transformation graphs and the $\mathcal{F}$-isolation number of the original graph? Or equality or inequality with other graph-theoretic parameters of the original graph?
\end{prob} 

\vspace{3mm}
\noindent \textbf{Declaration of competing interest}

The authors have no relevant financial or non-financial interests to disclose.

\noindent \textbf{Data availability}

No data was used for the research described in the article.

\noindent \textbf{Acknowledgements}

This work was supported by the National Science Foundation of China (Nos.12261074, 12461065 and 12371340).


\end{document}